\documentclass{amsart}
\usepackage{amsthm}
\usepackage{amsmath,amssymb}
\newtheorem{theorem}{Theorem}

\newtheorem{prop}[theorem]{Proposition}
\newtheorem{lem}[theorem]{Lemma}
\newtheorem{definition}[theorem]{Definition}

\numberwithin{theorem}{section}
\newtheorem{rem}[theorem]{Remark}
\allowdisplaybreaks
\newcommand{\Om} {\Omega}
\newcommand{\pa} {\partial}
\newcommand{\be} {\begin{equation}}
\newcommand{\ee} {\end{equation}}
\newcommand{\bea} {\begin{eqnarray}}
\newcommand{\eea} {\end{eqnarray}}
\newcommand{\Bea} {\begin{eqnarray*}}
\newcommand{\Eea} {\end{eqnarray*}}

\newcommand{\al} {\alpha}

\newcommand{\de} {\delta}

\newcommand{\De} {\Delta}
\newcommand{\la} {\lambda}

\def\R{{\mathbb R}}

\def\R{{\mathbb R}}

\newcommand{\na} {\nabla}

\newcommand{\rar}{\rightarrow}
\numberwithin{equation}{section}
\begin{document}
\title[qualitative remarks] {Some remarks on the qualitative questions for biharmonic equations}
\author[G. Dwivedi, J.\,Tyagi ]
{G.\,Dwivedi, J.Tyagi}
\address{G.\,Dwivedi \hfill\break
 Indian Institute of Technology Gandhinagar \newline
 Vishwakarma Government Engineering College Complex \newline
 Chandkheda, Visat-Gandhinagar Highway, Ahmedabad \newline
 Gujarat, India - 382424}
 \email{dwivedi\_gaurav@iitgn.ac.in}

\address{J.\,Tyagi \hfill\break
 Indian Institute of Technology Gandhinagar \newline
 Vishwakarma Government Engineering College Complex \newline
 Chandkheda, Visat-Gandhinagar Highway, Ahmedabad \newline
 Gujarat, India - 382424}
 \email{jtyagi@iitgn.ac.in, jtyagi1@gmail.com}
\date{20--12--2013}
\thanks{Submitted 20--12--2013.  Published-----.}
\begin{abstract}
In this article, we obtain several interesting remarks on the qualitative questions such as stability criteria,  Morse index, Picone's identity for biharmonic equations.
\end{abstract}
\maketitle
\section{introduction}
In the recent years there has been a good amount of interest on the existence and multiplicity of solutions to biharmonic equations.
Recently, A.E.Lindsay and J.Lega \cite{lin}
obtain multiple quenching solutions of a fourth order parabolic partial differential equation
\begin{equation}\label{pm}
\left\{
  \begin{array}{ll}
   u_{t}= - \Delta^2u+ \delta \Delta u -\la \frac{h(x)}{(1+u)^{2} } \,\,\,\,\mbox{in}\,\,\,\, \Omega\subset \R^{2},\\
   u=0=\frac{\partial u}{\partial \nu}\,\,\,\,\,\,\,\mbox{on} \,\,\,\partial\Omega,\\
  u=0,\,\,\,\,t=0,\, \,\,\,\,\,\,\,x \in \Omega.
 \end{array}
\right.
\end{equation}
Eq.\,\eqref{pm} models a microelectromechanical systems (MEMS) capacitor, where $u(x,\,t)$
represents the deflection of the device and $\delta$ represents the relative effects of tension and rigidity on the deflecting plate, $\la\geq 0$
represents the ratio of electric forces to elastic forces and $h$ represents possible heterogeneities in the deflecting surface's dielectric profile.
For the details on this subject, we refer the reader to \cite{pel}.
The steady state of Eq.\,\eqref{pm} (when $u(x,\,t)$ is independent of $t$) is:
\begin{equation}\label{pm1}
\left\{
 \begin{array}{ll}
   \Delta^{2} u - \delta \Delta u = - \la \frac{h(x)}{(1+u)^{2} } \,\,\,\,\mbox{in}\,\,\,\, \Omega, \\
   u=0=\frac{\partial u}{\partial \nu} \,\,\,\,\,\,\,\,\,\mbox{on} \,\,\,\partial\Omega.\\
 \end{array}
\right.
\end{equation}
For the existence of positive solutions to problems similar to \eqref{pm1} in $\R^{N},$ we refer to \cite{sat} and the references therein
and for the existence and bifurcation results to more general problem
\begin{equation}\label{pm4}
\left\{
 \begin{array}{ll}
   \Delta^{2} u - \Delta_{p} u = f(\la,\,x,\,u) \,\,\,\,\mbox{in}\,\,\,\, \Omega, \\
   u=0=\frac{\partial u}{\partial \nu} \,\,\,\,\,\,\,\,\,\mbox{on} \,\,\,\partial\Omega,\\
 \end{array}
\right.
\end{equation}
we refer to \cite{kan}. Equations of type \eqref{pm1} are also discussed on Riemannian manifold $(M^{n},g),\,\,n\geq 5,$  see \cite{you},
where the author obtain the existence of classical solutions to
\begin{equation}\label{pmo}
\Delta_{g}^{2} u - div(a(x) \nabla_{g} u)+ b(x) u = f(x)|u|^{N-2} u \,\,\,\,\mbox{in}\,\,\,\,M^{n}.\\
\end{equation}

There is  also a good amount of work on the qualitative questions such as
stability criteria, Picone's identity, Morse index, Sturm comparison theorem for Laplace as well as $p$-Laplace equations
but very little is known for biharmonic equations.
Recently, there have some investigations on the stability of solutions to $p$-Laplace equations/quasilinear elliptic equations,
see for instance, \cite{cas,karat,tyagi1,tyagi2} and the references therein.
We refer to \cite{ber} for the stability results to biharmonic equations and \cite{war} for Liouville theorems for stable radial solutions for the biharmonic operators.
Very recently, J.Wei and D.Ye  \cite{wei}
prove Liouville type results for stable solutions to the biharmonic problem
$$\Delta^2 u = u^q,\,\,\,\, u > 0 \,\,\,\mbox{in} \,\,\,\R^n\,\,\, \mbox{where}\,\, 1<q< \infty$$
and classify the unstable solutions for different ranges of $n$ and $q.$ In this context, there is a natural question to ask whether we can obtain the stability of
positive solution to biharmonic equations with sign changing nonlinerity.
In fact, motivated by the work of A.E.Lindsay and J.Lega \cite{lin},
we pose the stability question to the following fourth order boundary value problem:
\begin{equation}\label{pab}
\left\{
\begin{array}{ll}
    \Delta^2 u-\delta\Delta u=a(x)u-f_1(x,\,u)\,\,\,\mbox{in}\,\,\,\, \Omega,\\
    \delta u-2\Delta u\geq 0 \,\,\,\,\mbox{in}\,\,\,\Omega,\\
   u=0=\frac{\partial u}{\partial \nu}\,\,\,\,\,\mbox{on} \,\,\,\partial\Omega,
  \end{array}
\right.
\end{equation}
where $\Omega\subset\mathbb{R}^N$ is an open, smooth and bounded subset, $a\in L^{\infty}(\Omega), \,\,\delta>0$ and $f_1\in C(\overline{\Om}\times \R,\,\R).$

It is a well-known fact that in the qualitative theory of elliptic PDEs, Picone's identity plays an important role.
The classical Picone's identity says that
for differentiable functions $v>0$ and $u\geq 0,$
\be\label{pico}
|\nabla u|^2 + \frac{u^2}{v^2} |\nabla v|^2- 2 \frac{u}{v} \nabla u\nabla v= |\nabla u|^2- \nabla \left( \frac{u^2}{v} \right)\nabla v \geq 0.
\ee
\eqref{pico} has an enormous applications to second-order elliptic equations and systems, see for instance, \cite{al1,al2,al3,manes} and the references therein.
Let us write briefly the recent developments on Picone's identity.
In order to apply \eqref{pico} to p-Laplace equations, \eqref{pico} is extended by W. Allegretto and Y.X.Huang \cite{ale}.
The extension to \eqref{pico} is as follows:
\begin{theorem}\label{Pla}\cite{ale}
 Let $v>0$ and $u\geq0$ be differentiable.
Denote
 \begin{gather*}
 L(u,v)=|\nabla u|^p + (p-1) \frac{u^p}{v^p} |\nabla v|^{p}  - p \frac{u^{p-1}}{v^{p-1}} \nabla u |\nabla v|^{p-2} \nabla v.\\
 R(u,v)=|\nabla u|^p-\nabla(\frac{u^p}{v^{p-1} })|\nabla v|^{p-2}\nabla v.
 \end{gather*}
Then $L(u,v)=R(u,v).$ Moreover, $L(u,v)\geq 0$ and $L(u,v)=0$ a.e. in $\Omega$
if and only if $\nabla (\frac{u}{v})=0$ a.e. in $\Omega.$
\end{theorem}

Recently, the second author obtain a nonlinear analogue of \eqref{pico}
in \cite{tyagi} and obtained  some qualitative results. The nonlinear analogue of \eqref{pico} reads as follows:
\begin{theorem}\label{picnn}\cite{tyagi}
Let $v$ be a differentiable function in $\Om$ such that $v\neq 0$ in $\Om$ and $u$ be a non-constant differentiable function in $\Om.$
Let $f(y)\neq 0,\,\forall\,0\neq y\in \R$  and suppose that there exists $\alpha>0$  such that
$f'(y)\geq \frac{1}{\alpha},\,\forall\,0\neq y\in \R.$ Denote
\bea\label{gpic1}
L(u,\,v)= \alpha|\nabla u|^{2} - \frac{|\nabla u|^{2}}{f'(v)} + \left(\frac{u \sqrt{f'(v)}\nabla v }{f(v)} -
\frac{\nabla u}{\sqrt{f'(v)}}\right)^{2}.\\
R(u,\,v)= \alpha|\nabla u|^{2}- \nabla\left(\frac{u^2}{f(v)}\right)\nabla v.\label{gpic2}
\eea
Then $L(u,\,v)= R(u,\,v).$ Moreover, $L(u,\,v)\geq 0$ and $L(u,\,v)=0$ in $\Om$ if and only if  $u = c_1 v + c_2$ for some
arbitrary constants  $c_1,\,c_2.$
\end{theorem}
K. Bal \cite{bal} extended the nonlinear Picone's identity of \cite{tyagi} to
deal with p-Laplace equations. The extension reads as follows:

\begin{theorem}\label{PICN}\cite{bal}
 Let $v>0$ and $u\geq0$ be two non-constant differentiable functions
 in $\Omega$. Also assume that $f'(y)\geq (p-1)[f(y)^{\frac{p-2}{p-1}}]$
for all $y$. Define
 \begin{gather*}
 L(u,v)=|\nabla u|^p-\frac{p u^{p-1}\nabla u|\nabla v|^{p-2}\nabla v}{f(v)}
 +\frac{u^pf'(v)|\nabla v|^p}{[f(v)]^2}.\\
 R(u,v)=|\nabla u|^p-\nabla(\frac{u^p}{f(v)})|\nabla v|^{p-2}\nabla v.
 \end{gather*}
Then $L(u,v)=R(u,v)\geq0$. Moreover $L(u,v)=0$ a.e. in $\Omega$
if and only if $\nabla (\frac{u}{v})=0$ a.e. in $\Omega.$
\end{theorem}
There are also several interesting articles dealing with Picone's identity in different contexts. We just name a few articles, for instance,
for a Picone type identity to higher order half linear differentiable operators, we refer to \cite{jar} and the references therein, for
Picone identities to half-linear elliptic operators with $p(x)$-Laplacians, we refer to \cite{yosi} and for Picone-type identity to pseudo p-Laplacian with variable power,
we refer to \cite{bog}.
In \cite{dun}, D.R.Dunninger established a Picone identity for a class of fourth order elliptic differential inequalities. This identity says that if $u,\,v,\,a\De u,\,A\De v$
are twice continuously differentiable functions with $v(x)\neq 0$ and $a$ and $A$ are positive weights, then
\begin{align}\label{wp1}
&  div\left[ u\nabla(a\De u)- a \De u\nabla u - \frac{u^2}{v} \nabla (A \De v)+ A \De v. \nabla \left(\frac{u^2}{v}  \right)\right]  \nonumber\\
&= - \frac{u^2}{v} \De (A \De v)+ u \De (a \De u) + (A-a)(\De u)^2         \nonumber \\
& -  A \left( \Delta u - \frac{u}{v} \Delta v \right)^{2}+  A  \frac{2\Delta v}{v} \left(\nabla u - \frac{u}{v}\nabla v\right)^{2}.
\end{align}

In this context, threre is a natural question. Can we establish a nonlinear analogue of \eqref{wp1}?
More precisely, the aim of this article is twofold. Firstly, we establish a nonlinear analogue of Picone's identity which could deal with biharmonic equations and secondly using
the similar techniques, we consider the stability of a positive weak solution $u\in H_0^{2}(\Omega)\cap L^\infty (\Omega)$ of \eqref{pab} in any arbitrary smooth bounded domain for sign changing nonlinearity.
In this paper, we assume that $\Omega\subset\mathbb{R}^N$ is an open, smooth and bounded subset and $a\in L^{\infty}(\Omega)$ are such that
\eqref{pab} has a positive weak solution $u\in H_0^{2}(\Omega)\cap L^\infty (\Omega).$

We make the following hypothesis on the nonlinearity $f_1$:

(H1) Let $f_1\in C(\overline{\Omega}\times\mathbb{R}, \mathbb{R})$ and $C^{1}$ in the $y$ variable and satisfies
\[\frac{\partial f_{1}(x,y)}{\partial y}\geq\frac{f_1(x,y)}{y},\,\,\,\,\,\forall\,\,\, 0< y\in\mathbb{R},\,\,\,\forall\,\,x\in \overline{\Om}.\]
The plan of this paper is as follows. Section 2 deals with  the nonlinear analogue of Picone's identity which could deal with biharmonic equations.
In Section 3, we give several applications of Picone's identity to biharmonic equations. In Section 4, we establish a stability theorem of a positive weak solution to
\eqref{pab}. 

\section{nonlinear analogue of Picone's identity}
In this section, we establish a nonlinear analogue of Picone's identity. The next lemma can be obtained from \eqref{wp1} with some assumptions.
Since the proof is short and interesting so we write it independently here with useful insights.

\begin{lem} \rm{(Picone's identity)}\label{pic}
 Let $u$ and $v$ be twice continuously differentiable functions in $\Om$ such that $v>0,\,\,-\De v>0$ in $\Om.$ Denote
  \begin{eqnarray*}
  L(u,\,v)&=& \left( \Delta u - \frac{u}{v} \Delta v \right)^{2} - \frac{2\Delta v}{v} \left(\nabla u - \frac{u}{v}\nabla v\right)^{2}.\\
 R(u,\,v) &= &|\De u|^{2}- \De\left(\frac{u^2}{v} \right) \De v.
\end{eqnarray*}
 Then \rm{(i)}\,\,$L(u,\,v)= R(u,\,v)$\,\,\,\rm{(ii)}\,\,$L(u,\,v)\geq 0$ and \rm{(iii)}\,\,$L(u,\,v)=0$ in $\Om$ if and only if $u= \alpha v $ for some $\alpha \in \R.$
\end{lem}
\begin{proof}
 Let us expand $R(u,\,v):$
 \begin{eqnarray*}
 R(u,\,v) &= &|\De u|^{2}- \De\left(\frac{u^2}{v} \right) \De v\\
 &=& |\De u|^{2}+ \frac{u^2}{v^2} |\De v|^{2} - \frac{2u}{v} \De u \De v - \frac{2}{v} |\nabla u|^{2} \De v
 + \frac{4u}{v^2} \nabla u\nabla v\De v - \frac{2 u^2}{v^3} |\nabla v|^{2} \De v\\
&=& \left( \Delta u - \frac{u}{v} \Delta v \right)^{2} - \frac{2\Delta v}{v} \left(\nabla u - \frac{u}{v}\nabla v\right)^{2}\\
&=& L(u,\,v),
\end{eqnarray*}
which proves the first part. Now using the fact that $v>0,\,\,-\De v> 0$ in $\Om,$ one can see that $L(u,\,v)\geq 0$ and therefore \rm{(ii)} is proved.
Now $L(u,\,v)=0$ in $\Om$  implies that
\begin{equation*}
0  =  \left( \Delta u - \frac{u}{v} \Delta v \right)^{2} - \frac{2\Delta v}{v} \left(\nabla u - \frac{u}{v}\nabla v\right)^{2},\,\,\,\,\mbox{i.e},
\end{equation*}
 $$ 0\leq - \frac{2\Delta v}{v} \left(\nabla u - \frac{u}{v}\nabla v\right)^{2}=- \left( \Delta u - \frac{u}{v} \Delta v \right)^{2}\leq 0,$$
which implies that there exists some $\al \in \R$ such that $u = \al v.$ Conversely, when $u = \al v,$  one can see easily that $L(u,\,v)=0,$ and therefore \rm{(iii)} is proved.
\end{proof}
\begin{rem}
 We note that the above lemma also holds if we replace $v>0$ and $-\De v>0$ in $\Om$ by $v<0$ and $-\De v<0$ in $\Om,$ respectively.
\end{rem}
In the next proposition, we establish a nonlinear analogue of Picone's identity for biharmonic equations.

\begin{prop}\rm{(Nonlinear analogue of Picone's identity)}\label{npic}
Let $u$ and $v$ be twice continuously differentiable functions in $\Om$ such that $v>0,$ $-\De v>0$ in $\Om$. Let $f\colon \R \rar (0,\infty)$ be a $C^2$
function such that $f''(y)\leq 0,\,\,f'(y) \geq 1,\, \forall \,0\neq y\in\R.$ Denote
\begin{eqnarray*}
 L(u,\,v)&=& |\De u|^{2}-\frac{|\De u|^2}{f'(v)}+\left( \frac{\Delta u}{\sqrt{f'(v)}} - \frac{u}{f(v)}\sqrt{f'(v)}\De v \right)^{2}\\
 & \,\,\,\,&  -\frac{2\Delta v}{f(v)} \left(\nabla u - \frac{uf'(v)}{f(v)}\nabla v\right)^{2}+\frac{u^2f''(v)}{f(v)}|\na v|^2\De v.\\
 R(u,\,v) &= &|\De u|^{2}- \De\left(\frac{u^2}{f(v)} \right) \De v.
\end{eqnarray*}
 Then \rm{(i)}\,\,$L(u,\,v)= R(u,\,v)$\,\,\,\rm{(ii)}\,\,$L(u,\,v)\geq 0$ and \rm{(iii)}\,\,$L(u,\,v)=0$ in $\Om$ if and only if $u= c v +d$ for some $c,d \in \R.$
\end{prop}
\begin{proof}
Let us expand $R(u,v):$
\begin{eqnarray*}
R(u,\,v) &= &|\De u|^{2}- \De\left(\frac{u^2}{f(v)} \right) \De v\\
&=&  |\De u|^2-\frac{|\De u|^2}{f'(v)}+\left(\frac{|\De u|^2}{f'(v)}+\frac{u^2f'(v)}{f^2(v)}|\De v|^2-\frac{2u\De u\De v}{f(v)}\right)\\
& \,\,& -\frac{2\De v}{f(v)}\left(|\na u|^2+\frac{u^2f'^2(v)}{f^2(v)}|\na v|^2-\frac{2uf'(v)}{f(v)}\na u. \na v\right)+\frac{u^2f''(v)}{f^2(v)}|\na v|^2\De v \\
&=& |\De u|^{2}-\frac{|\De u|^2}{f'(v)}+\left( \frac{\Delta u}{\sqrt{f'(v)}} - \frac{u}{f(v)}\sqrt{f'(v)}\De v \right)^{2}\\ & \,\,\,\,&
-\frac{2\Delta v}{f(v)} \left(\nabla u - \frac{uf'(v)}{f(v)}\nabla v\right)^{2}+\frac{u^2f''(v)}{f^2(v)}|\na v|^2\De v\\
&=& L(u,v),
\end{eqnarray*}
which proves the first part. Now using the fact that $-\De v> 0,\,f'(y) \geq 1,\,$ and $ f''(y)\leq 0,\,\forall \,0\neq y\in\R,$
we get  $L(u,v)\geq 0$ and therefore \rm{(ii)}
is proved. Now $L(u,v)=0$ in $\Om$ implies that
\begin{equation}\label{conv}
|\De u|^2-\frac{|\De u|^2}{f'(v)}=0\,\,\text{and}\,\, \na u-\frac{uf'(v)}{f(v)}\na v=0.
\end{equation}
This gives $f'(v)=1$ or $f(v)=v+c_1\,\text{where}\, c_1 $ is a constant, which yields
\begin{equation*}
(\na u)(v+c_2)-u\na(v+c_2)=0 \, \text{or}\, \na\left(\frac{u}{v+c_2}\right)=0 \,\,\,\text{i.e.,}\,\,\, \,u=cv+d
\end{equation*}
for some constants $c$ and $d.$
Conversely, let us assume \eqref{conv} holds. We need to show that $L(u,v)=0.$ From \eqref{conv}, we get that $f'(v)=1$ and therefore $f''(v)=0.$ Now it remains to show that
\begin{equation*}
\left( \frac{\Delta u}{\sqrt{f'(v)}} - \frac{u}{f(v)}\sqrt{f'(v)}\De v \right)=0\,\,\,\text{i.e.,}\,\, f(v)\De u=uf'(v)\De v.
\end{equation*}
From \eqref{conv}, we get
\begin{eqnarray*}
0 &=& f(v) \na u -uf'(v) \na v \\
0 &=& f(v)\De v +f'(v) \na u.\na v-f'(v)\na u.\na v-uf''(v) |\na v|^2-uf'(v)\De v\\
f(v)\De u &=& uf'(v)\De v,
\end{eqnarray*}
which completes the proof.
\end{proof}

\section{applications}
This section deals with the applications of Lemma\,\ref{pic} and Proposition\,\ref{npic}.
In next theorem,  we obtain a Hardy-Rellich type inequality.
\begin{theorem}
Assume that there is a $C^2$ function $v$ satisfying
\begin{equation}\label{hrd}
\De^2 v \geq \la g f(v),\,\, v>0,\, -\De v>0\,\, \text{in}\,\, \Om,
\end{equation}
for some $\la>0$ and a nonnegative continuous function $g$ on $\Om$ and $f$ satisfies the conditions of  Proposition\,\ref{npic}. Then for any $u\in C_c^{\infty} (\Om)$
\begin{equation}\label{app1}
\int_\Om |\De u|^2 dx \geq \la\int_\Om g|u|^2dx.
\end{equation}
\end{theorem}
\begin{proof}
Take $\phi \in C_c^{\infty}(\Om)$, by Proposition\,\ref{npic}, we have
\begin{eqnarray*}
0 &\leq & \int_\Om L(\phi,v)dx=\int_\Om R(\phi,v) dx\\
& =& \int_\Om  |\De \phi |^2dx-  \int_\Om \De \left(\frac{\phi^2}{f(v)}\right) \De vdx\\
&=&  \int_\Om  |\De \phi |^2dx-  \int_\Om (\De^2 v). \frac{\phi^2}{f(v)}dx,\,\,\,\,\text{on integration},\\
&\leq&  \int_\Om  |\De \phi |^2dx- \la \int_\Om \phi^2 gdx\,\,\,\,\,\,\text{by \eqref{hrd}}.
\end{eqnarray*}
Letting $\phi \rar u,$ we have
\begin{equation*}
\int_\Om |\De u|^2 dx \geq \la\int_\Om g|u|^2dx.
\end{equation*}
\end{proof}
The next lemma deals with a necessary condition for the nonnegative solutions of biharmonic equations.

\begin{lem}\label{l1}
Let $u\in H^{2}(\Om)\cap H_{0}^{1}(\Om)$ be a nonnegative weak solution (not identically zero) of
 \begin{equation}\label{ev1}
\De^{2} u =  a(x) u\,\,\,\mbox{in}\,\,\Omega,\,\,\,\,u= \De u = 0\,\,\,\mbox{on}\,\,\partial\Om,
\end{equation}
where $0\leq a\in L^{\infty}(\Om),$ then $-\De u>0$ in $\Om.$
\end{lem}

\begin{proof}
 Let $-\De u =v.$ Then writing \eqref{ev1} into system form, we get
\begin{equation}\label{sy}
\left\{
\begin{array}{ll}
    -\Delta u = v\,\,\,\mbox{in}\,\,\,\, \Omega,\\
    -\Delta v= a(x) u\,\,\,\mbox{in}\,\,\,\,\Omega,\\
   u=0=v\,\,\,\,\,\mbox{on} \,\,\,\partial\Omega,
  \end{array}
\right.
\end{equation}
Since $a(x)\geq 0$ in $\Om,$ so by maximum principle, we get $v\geq 0.$ By strong maximum principle, either $v>0$ or $v\equiv 0$ in $\Om.$
If $v\equiv 0,$ then we have
$$ -\Delta u = 0\,\,\,\mbox{in}\,\,\,\, \Omega;\,\,\,v= 0\,\,\,\mbox{on} \,\,\partial \Om.$$
Again by maximum principle, we get $u\equiv 0,$ which is a contradiction and therefore $v>0$ in $\Om$ and hence
$$-\Delta u > 0\,\,\,\mbox{in}\,\, \Omega.  $$
 \end{proof}

Next, we consider the following singular system of fourth order elliptic equations:
\begin{equation}\label{singsys}
\begin{aligned}
  \De ^2u &= f(v)\,\,\text{in}\, \Om ,\\
  \De ^2v &= \frac{\left(f(v)\right)^2}{u}\,\,\text{in}\, \Om,\\
  u>0, &\, v>0 \,\, \,\,\,\text{in}\, \Om,\\
  u= \De u=&0 = v =\De v\,\,\, \text{on}\,\, \pa \Om,
\end{aligned}
\end{equation}
where $f$ is as defined in Proposition\,\ref{npic}.
In the next theorem, we show a linear relationship between the components $u$ and $v,$ where $(u,\,v)$ is a solution of
\eqref{singsys}.
\begin{theorem}
Let $(u,v)$ be a weak solution of \eqref{singsys} and $f$ satisfy the conditions of Proposition\,\ref{npic}. Then $u=c_1 v+c_2$, where $c_1,c_2$ are constants.
\end{theorem}
\begin{proof}
Let $(u,v)$ be a weak solution of \eqref{singsys}. Then
\begin{equation}\label{1}
\int_\Om \De u \De \phi_1 dx=\int_\Om f(v) \phi_1  dx.
\end{equation}
\begin{equation}\label{2}
\int_\Om \De v \De \phi_2 dx =\int_\Om \frac{f^2(v)}{u} \phi_2 dx
\end{equation}
hold for any $\phi_1,\phi_2 \in H^2(\Om)\cap H_0^1(\Om).$ Now choosing $\phi_1 =u$ and $\phi_2=\displaystyle\frac{u^2}{f(v)}$ in \eqref{1} and \eqref{2}, respectively, we obtain
\begin{equation*}
\int_\Om |\De u|^2 dx =\int_\Om f(v) u dx=\int_\Om \De v \De \left(\frac{u^2}{f(v)}\right) dx.
\end{equation*}
Hence we have
\begin{equation*}
\int_\Om R(u,v) dx=\int_\Om \left[|\De u|^2- \De v \De \left(\frac{u^2}{f(v)}\right)\right] dx=0.
\end{equation*}
By positivity of $R(u,v),$ we get $R(u,v)=0$ and by Lemma\,\ref{l1}, we have
$$-\De u>0,\,\,\,\,-\De v>0\,\,\,\,\,\text{in}\,\,\Om.$$
Now an application of Proposition \ref{npic} yields that $u=c_1v+c_2$
for some constants $c_1$ and $c_2.$
\end{proof}

Let us consider the following eigenvalue problem
\begin{equation}\label{ev}
\De^2 u = \lambda a(x) u\,\,\,\mbox{in}\,\,\Omega,\,\,\,\,u= \De u = 0\,\,\,\mbox{on}\,\,\partial\Om,
\end{equation}
where $\Om\subset \R^{N}$ is an open, bounded subset and $0\leq a\in L^{\infty}(\Om).$
We recall that a value $\la\in \R$ is an eigenvalue of \eqref{ev} if and only if there exists $u\in H^{2}(\Om)\cap H_{0}^{1}(\Om)/\{0\}$ such that
\begin{equation}
 \int_{\Om} \De u. \De \phi dx = \int_{\Om} a(x) u \phi dx,\,\,\,\forall\,\,\phi \in H^{2}(\Om)\cap H_{0}^{1}(\Om)
\end{equation}
and $u$ is called an eigenfunction associated with $\la.$ The least positive eigenvalue of \eqref{ev}  is defined as
\begin{equation}\nonumber
 \la_{1}= \inf\left\{ \int_{\Om} |\De u|^{2} dx:\,\,\,\,u\in H^{2}(\Om)\cap H_{0}^{1}(\Om)\,\,\,\,\mbox{and} \,\,\,\int_{\Om} a(x) |u|^{2} dx=1     \right\}.
\end{equation}

\begin{lem}
$\la_{1}$ is attained.
\end{lem}
\begin{proof}
For showing the above infimum is attained, let us introduce the functionals
$$J,\,\,G:H^{2}(\Om)\cap H_{0}^{1}(\Om)\longrightarrow \R\,\,\,\,\,\,\mbox{defined by}$$
$$ J(u) = \frac{1}{2}  \int_{\Om} |\De u|^{2} dx,\,\,\,\,G(u)=   \frac{1}{2} \int_{\Om} a(x) |u|^{2} dx,\,\,\,\,u\in H^{2}(\Om)\cap H_{0}^{1}(\Om).$$
It is easy to see that $J$ and $G$ are $C^{1}$ functionals. By definition, $\la\in \R$ is an eigenvalue of \eqref{ev} if and only if
there exists $u\in H^{2}(\Om)\cap H_{0}^{1}(\Om)/\{0\}$ such that
$$J'(u) = \la G'(u).$$
Le us define
$$ M=   \left\{ u\in  H^{2}(\Om)\cap H_{0}^{1}(\Om) |\frac{1}{2} \int_{\Om} a(x) |u|^{2} dx=1 \right\}.$$
Since $a\geq 0$ so $M \neq \emptyset$ and $M$ is a $C^{1}$ manifold in $H^{2}(\Om)\cap H_{0}^{1}(\Om).$
It is also easy to see that $J$ is coercive and (sequentially) weakly lower semicontinuous on $M$ and $M$ is a weakly closed subset of $H^{2}(\Om)\cap H_{0}^{1}(\Om).$
Now by an application of Theorem 1.2\,\,\cite{str},   $J$ is bounded from below on $M$
and attains its infimum in $M.$ Also by Lagrange's multiplier rule
$$  J'(u) = \la_{1} G'(u)      $$
and therefore $\la_{1}$ is attained.
\end{proof}

In the next lemma, we show that the first eigenfunction $u$ corresponding to the first eigenvalue $\la_1$ of \eqref{ev} is of one sign.
We use the following theorem.

\begin{theorem}\cite{gazz} (Dual cone decomposition theorem)\label{dc}\,\,\,Let $H$ be a Hilbert space with scalar product $(\,\,., )_{H}.$ Let $K\subset H$ be a  closed, convex nonempty cone. Let $K^{*}$
 be its dual cone, namely
 $$ K^{*}= \{w\in H |   (w,\,v)_{H} \leq 0,\,\,\,\forall\,\,v\in K \}.$$
Then for anu $u\in H,$ there exists a unique $(u_{1},\,u_{2})\in K\times K^{*}$ such that
\begin{equation}\label{dec}
 u= u_{1} + u_{2},\,\,\,\,(u_{1},\,u_{2})_{H}=0.
\end{equation}
In particular,$$||u ||^{2}_{H}=   ||u_1 ||^{2}_{H} +   ||u_2 ||^{2}_{H}.$$
 Moreover, if we decompose arbitrary $u,\,v\in H$ according to \eqref{dec}, i.e.,
 $$ ||u-v ||^{2}_{H}   \geq || u_{1} - v_{1}||_{H}^{2} +   || u_{2} - v_{2}||_{H}^{2}. $$
 In particular, the projection onto $K$ is Lipschitz continuous.
\end{theorem}

\begin{lem}\label{l2}
The eigenfunction $u$ corresponding to the first eigenvalue $\la_1$ of \eqref{ev} is of one sign.
\end{lem}
\begin{proof}
 Using Theorem\ref{dc}, and classical maximum principle for $-\De$,  Ferrero et al. \cite{far} obtain the positivity of the minimizers of the problem
$$ S_{p}= \min_{w\in X/\{0\}} \frac{||\De w   ||_{2}^{2}  }{|| w||_{p}^{2}} ,\,\,\, 1 \leq p < \frac{2n}{n-4},$$
where $X=H^{2}(B)\cap H_{0}^{1}(B),\,\,B$ denotes the unit ball in $\R^n.$ The same proof works for eigenfunction $u$ corresponding to the first eigenvalue $\la_1$ of \eqref{ev}
in $\Om.$ For this, we refer to \cite{far} and omit the details.
\end{proof}

Next we show the strict monotonicity of the principle eigenvalue $\la_1.$
\begin{theorem}
Suppose $\Om_1 \subset\Om_2$ and $\Om_1 \neq \Om_2.$ Then $\la_1(\Om_1) > \la_1(\Om_2),$ if both exist.
\end{theorem}
\begin{proof}
Let $u_i$ be a positive eigenfunction associated with $\la_1(\Om_i),\, i=1,2.$ For $\phi\in C_c^\infty(\Om_1)$,
\begin{eqnarray}
\label{sti}
0 &\leq & \int_{\Om_1}L(\phi,u_2)\, dx=\int_{\Om_1} R(\phi,u_2)\,dx \nonumber\\
&=& \int_{\Om_1}\left(|\De \phi|^2-\De(\frac{\phi^2}{u_2})\De u_2\right)\, dx\nonumber \\
& = & \int_{\Om_1} |\De \phi|^2 dx-\int_{\Om_1}\frac{\phi^2}{u_2}\De^2u_2\, dx\nonumber\\
&=& \int_{\Om_1} |\De\phi|^2\, dx-\la_1(\Om_2) \int_{\Om_1}a(x)\phi^2 \, dx
\end{eqnarray}
Letting $\phi\rar u_1$ in \eqref{sti}, we obtain
\[0\leq  \int_{\Om_1} L(u_1,u_2)dx=(\la_1(\Om_1)-\la_1(\Om_2)) \int_{\Om_1}a(x)u_1^2\, dx.\]
This gives $\la_1(\Om_1)-\la_1(\Om_2)\geq 0.$ Now if $\la_1(\Om_1)-\la_1(\Om_2)=0$ then $L(u_1,u_2)=0$ and an application of Lemma\,\ref{pic} implies that
 $u_1=cu_2,$ which is not possible as $\Om_1\subset\Om_2$   and $\Om_1\neq \Om_2.$
This completes the proof.
\end{proof}
In the next theorem, using Picone's identity (Lemma\,\ref{pic}), we show that $\la_1$ is simple, i.e.,
the eigenfunctions associated to it are a constant multiple of each other.
\begin{theorem}
 $\la_{1}$ is simple.
\end{theorem}
\begin{proof}
 Let $u$ and $v$ be two eigenfunctions associated with $\la_{1}.$ From Lemma\,\ref{l2}, without any loss of generality, we can assume that
 $u$ and $v$ are positve in $\Om.$ Now by Lemma\,\ref{l1}, we have
 $$-\De u>0,\,\,\,\,-\De v>0  \,\,\,\,\mbox{in} \,\,\Om.          $$
 Let $\epsilon>0.$ From Lemma\,\ref{pic}, we have
\begin{align}\label{w1}
& 0\leq \int_\Omega L(u,\,v+ \epsilon) dx \nonumber\\
&=\int_\Omega  R(u,\,v+ \epsilon) dx \nonumber \\
&= \int_{\Om} \left[|\De u|^{2}  - \De \left(\frac{u^2}{v+\epsilon} \right) \De v \right]dx \nonumber    \\
&= \la_{1} \int_{\Om } a(x) u^{2} dx- \int_{\Om} \De \left(\frac{u^2}{v+\epsilon} \right) \De v dx.
\end{align}
The function $\phi= \frac{u^2}{v+\epsilon}$ and is admissible in the weak formulation of
$$  \De^{2}v = \la_{1} a(x) v, \,\,\,\,\mbox{i.e.,}                   $$
\begin{equation}\label{w2}
\int_{\Om} \De v \De \left(\frac{u^2}{v+\epsilon} \right) dx= \la_{1} \int_{\Om} a(x) v \left(\frac{u^2}{v+\epsilon}\right) dx.
 \end{equation}
From \eqref{w1} and \eqref{w2}, we get
 \begin{align*}
& 0\leq \int_\Omega L(u,\,v+ \epsilon) dx \nonumber\\
&=\la_{1}  \int_\Omega  a(x) \left[u^{2} - v  \left( \frac{u^2}{v+\epsilon}\right)     \right] dx.    
\end{align*}
Letting $\epsilon\longrightarrow 0,$ in the above inequality, we get
 $$ L(u,\,v) = 0        $$
and again by an application of Lemma\,\ref{pic}, there exists $\al\in \R$ such that $$u= \al\,v,$$
which proves the simplicity of $\la_{1}.$
 \end{proof}
Next, we show the sign changing nature of any eigenfunction $v$ associated to a positive eigenvalue $0<\la\neq \la_{1}.$
\begin{prop}
Any eigenfunction $v$ associated to a positive eigenvalue $0<\la\neq \la_{1}$ changes sign.
\end{prop}
\begin{proof}
Assume by contradiction that $v\geq 0,$ the case $v\leq 0$ can be dealt similarly. By Lemma\,\ref{l1}, $v>0$ in $\Om.$ Let $\phi>0$
be an eigenfunction associated with $\la_{1}>0.$ For any $\epsilon>0,$ we apply Lemma\,\ref{pic} to the pair $\phi,\,v+\epsilon$
and get
\begin{align}\label{wwv1}
& 0\leq  \int_\Omega L(\phi,\,v+\epsilon)  dx \nonumber\\
&=  \int_\Omega R(\phi,\,v+\epsilon) dx \nonumber \\
&= \int_{\Om} \left[|\De \phi|^{2} - \De \left( \frac{\phi^{2} }{v+\epsilon} \right)  \De v\right] dx \nonumber \\
&= \int_{\Om} \left[\la_{1} a(x) \phi^{2} - \De \left( \frac{\phi^{2} }{v+\epsilon} \right)  \De v  \right] dx.
\end{align}
For every $\phi\in H^{2} (\Om) \cap H_{0}^{1}(\Om),\,\,\frac{\phi^2}{v+\epsilon}\in H^{2} (\Om) \cap H_{0}^{1}(\Om) $
and is admissible in the weak formulation of
 $$\De^{2} v= \la a(x) v\,\,\mbox{in}\,\,\Om;\,\,\,v= \De v =0\,\,\,\,\mbox{on}\,\,\partial\Om.      $$
 This implies that
 \begin{equation}\label{23}
  \int_{\Om}  \De v \De \left( \frac{\phi^{2} }{v+\epsilon} \right)dx= \la \int_{\Om} a(x) v\frac{\phi^{2} }{v+\epsilon} dx.
 \end{equation}
From \eqref{wwv1} and \eqref{23}, we get
 \begin{align}
 0\leq  \int_\Omega  \left[\la_{1} a(x) \phi^{2}-  \la a(x) v\frac{\phi^{2} }{v+\epsilon}\right]dx. \nonumber
\end{align}
 Letting $\epsilon\longrightarrow 0$ in the above inequality, we get
\begin{eqnarray*}\label{u1}
0 &\leq  (\la_{1} - \la)   \int_\Omega a(x) \phi^{2} dx, \nonumber
\end{eqnarray*}
 which is a contradiction, because $ \int_\Omega a(x) \phi^{2} dx>0$
 and hence $v$ must change sign.
\end{proof}
For the application of Lemma\,\ref{pic} on Morse index, let us consider the following boundary value problem

\begin{equation}\label{s1}
 \De^{2}u = a(x) G(u)\,\,\,\mbox{in}\,\,\Om;\,\,\,\,u= \De u = 0\,\,\,\mbox{on} \,\,\partial\Om.
\end{equation}
For the existence of positive solution to the equations similar to \eqref{s1}, we refer the reader to \cite{gon}. By the standard elliptic regularity theory,
$u \in C^{4}(\Om)\cap C^{3}(\bar{\Om}).$
We shall assume that there exists a positive $C^4$ solution $u$ of the boundary value
problem \eqref{s1}.
For the solution $u \in C^{4}(\Om),$ the Morse index is defined via the eigenvalue problem for the linearization at $u:$
\begin{definition}
{\bf Morse index:}\,\,The Morse index of a solution $u$ of \eqref{s1} is the number of negative eigenvalues of the linearized operator
\be
\De^{2}- a(x) G'(u)
\ee
acting on $H^{2} (\Om) \cap H_{0}^{1}(\Om),$ i.e.,
the number of eigenvalues $\la$  such that $\la<0,$ and the boundary value problem
\be
\De^{2} w- a(x) G'(u)w= \la w\,\,\,\mbox{in}\,\,\,\,\Om;\,\,\,w=0= \De w\,\,\mbox{on}\,\,\partial\Om
\ee
has a nontrivial solution $w$ in $H^{2} (\Om) \cap H_{0}^{1}(\Om).$
\end{definition}
The next theorem gives an application of Lemma\,\ref{pic}.
\begin{theorem}
Let us consider\,\eqref{s1}. Let $a\in C^{\alpha}(\Om),\,\,\,0<\alpha<1$ and $G\in C^{1}(\R,\,\R)$ be
such that
$$ \frac{G(v)}{v} \geq G'(0)\geq 0,\,\,\,\forall\,\,0<v\in \R.$$
Then the trivial solution  of \eqref{s1} has Morse index $0.$
\end{theorem}
\begin{proof}
 Let $v$ be a positive weak solution of \eqref{s1}. Then
\begin{equation}\label{es1}
 \int_{\Om}\De v \De \psi dx = \int_{\Om} a(x) G(v) \psi dx,\,\,\forall\,\psi\in H^{2} (\Om) \cap H_{0}^{1}(\Om).
\end{equation}
For any $ w\in C_{c}^{\infty}(\Om),$  let us take $\frac{w^2}{v}$ as a test function in \eqref{es1} and obtain
\begin{equation}\label{es2}
 \int_{\Om}\De v \De \left(\frac{w^2}{v} \right)dx = \int_{\Om} a(x)\frac{G(v)}{v} w^2 dx.
\end{equation}
Since $v$ is a positive solution of \eqref{s1} so using the fact that $G(v)\geq 0$ and in view of Lemma\,\ref{l1} , one can see that
 $$-\De v>0.      $$
Now an application of Lemma\,\ref{pic} for $u=w$  yields that
\begin{align}\label{ww1}
&  \int_\Omega |\De w|^{2} dx \nonumber\\
&\geq \int_\Omega \De v \De \left(\frac{w^2}{v} \right)  dx \nonumber \\
&= \int_{\Om}   a(x) \frac{G(v)}{v} w^2 dx \nonumber    \\
&\geq  \int_{\Om } a(x)G'(0)  w^2 dx.
\end{align}
Consider the eigenvalue problem associated with the linearization for \eqref{s1} at $0,$ which is
 \begin{equation}\label{ww2}
 \De^{2} w- a(x) G'(0)w= \la w\,\,\,\mbox{in}\,\,\,\,\Om;\,\,\,w=0= \De w\,\,\mbox{on}\,\,\partial\Om.
 \end{equation}
By the variational characterization of the eigenvalue in \eqref{ww2}, from \eqref{ww1}, one can see that $\la\geq 0,$
which proves the claim.
\end{proof}

\section{Stability of positive solutions}

In this section, we consider the stability of a positive solution to \eqref{pab}. The functional associated with \eqref{pab} is
$$E\colon H_0^{2}(\Omega) \longrightarrow \mathbb{R}$$
defined by
\[E(u)=\frac{1}{2}\int_\Omega|\Delta u|^2 dx+\frac{\de}{2}\int_\Omega|\nabla u|^2 dx-\frac{1}{2}\int_\Omega a(x)u^2 dx+\int_\Omega F_1(x,\,u)dx,\]
where
\[F_1(x,s)=\int_0^s f_1(x,t)dt.\]

The weak formulation of \eqref{pab} is the following:
\begin{equation}\label{we}
\int_\Omega \Delta u\Delta\phi dx+\delta\int_\Omega \nabla u.\nabla\phi dx=\int_\Omega a(x)u\phi dx-\int_\Omega f_1(x,u)\phi dx,~~~~~~~\forall~\phi~C_c^2(\Omega).
\end{equation}
By the classical elliptic regularity theory, $u\in C^{3}(\overline{\Omega}),$  see Theorem\,2.20 \cite{gazz} when $\delta=0$ in \eqref{pab} and in
fact the same proof works in case $\delta\neq 0.$
Therefore, we assume that the solution of \eqref{pab} belongs to $  C^{3}(\overline{\Omega}).$
The linearized operator $L_u$ associated with \eqref{pab} at a given solution $u$ is defined by following duality:
\[ L_u: v\in H_0^2(\Omega) \longrightarrow L_{u}(v)\in(H_0^2(\Omega))',\]
where
\[L_{u}(v): \psi \in H_0^2(\Omega) \longrightarrow L_{u}(v,\psi)\]
and
\[L_u(v,\psi)= \int_\Omega\Delta v \Delta\psi dx+\delta \int_\Omega\nabla v.\psi dx-\int_\Omega a(x)v\psi dx+\int_\Omega \frac{\partial f_1(x,u)}{\partial u} v \psi dx.\]

It is easy to see that $L_u$ is well-defined and the first eigenvalue of $L_u$ is given by
\be\label{pev}
\la_1= \inf_{v\in H_{0}^{2}(\Om),\,v\neq 0} \frac{L_{u}(v,\,v)} {\int_{\Om}v^2 dx }.
\ee

We say that the solution $u$ of \eqref{pab} is stable if
\begin{equation}\label{stb}
\int_\Omega|\Delta v|^2 dx+\delta \int_\Omega|\nabla v|^2 dx-\int_\Omega a(x)v^2 dx+\int_\Omega \frac{\partial f_1(x,u)}{\partial u} v^2 dx \geq 0
\end{equation}
for every $v\in~C_c^2(\Omega),$ see \cite{wei2} for the definition of stabiity of solutions to biharmonic problems.
Actually, \eqref{stb} implies that the principal eigenvalue of the linearized equation
associated with \eqref{pab} is nonnegative and hence the solution $u$ of \eqref{pab}  is stable.

Next, we state and prove the stability theorem.

\begin{theorem}\label{th1}
 Let $u\in H_0^{2}(\Omega)\cap L^\infty (\Omega)$ be a positive  solution to \eqref{pab} in $\Om.$ Let \rm{(H1)} hold. Then $u$ is stable.
\end{theorem}
\begin{proof}
Since $u\in H_0^{2}(\Omega)\cap L^\infty (\Omega)$ be a positive  solution of \eqref{pab} so by the classical elliptic regularity theory, $u\in C^{2}(\overline{\Omega}).$
Now for any $v\in C_c^2(\Omega),$ we choose
\[\phi=\frac{v^2}{u}\]
as a test function in \eqref{we}. Since
\[\nabla \phi =\frac{2uv\nabla v-v^2\nabla u}{u^2},\] and
\[\Delta \phi=\frac{2u^3 |\nabla v|^2-4 v u^2 \nabla u.\nabla v+2v^2 u |\nabla u|^2+ 2vu^3 \Delta v- v^2u^2 \Delta u}{u^4}\]
so from \eqref{we}, we get
\begin{align*}\int_\Omega\Delta u.\left[\frac{2u^3 |\nabla v|^2-4 v u^2 \nabla u.\nabla v+2v^2 u |\nabla u|^2+ 2vu^3 \Delta v- v^2u^2 \Delta u}{u^4}\right] dx &\\
 +\delta\int_\Omega \nabla u.\left[\frac{2uv\nabla v-v^2\nabla u}{u^2}\right]dx = \int_\Omega a(x)v^2dx-\int_\Omega \frac{f_1(x,u)v^2}{u}dx.
\end{align*}
This yields that

\begin{align*}
& \int_\Omega \frac{-4v}{u^2}\Delta u\nabla u.\nabla vdx+ \int_\Omega \frac{2v}{u} \Delta u\Delta vdx+\int_\Omega \frac{2}{u}|\nabla v|^2 \Delta udx -
\int_\Omega\frac{v^2}{u^2}|\Delta u|^2dx    \\
& + \int_\Omega \frac{2 v^2}{u^3}|\nabla|^2\Delta udx +
\delta\int_\Omega \nabla u\left[ \frac{2v \nabla v}{u}-\frac{v^2}{u^2}\nabla u\right]dx-\underbrace{\int_\Omega a(x) v^2dx} \\
& + \underbrace{\int_\Omega f_1(x,u)\frac{v^2}{u}dx}+\underbrace{\int_\Omega |\Delta v|^2 dx}-\int_\Omega |\Delta v|^2dx +\underbrace{ \delta\int_\Omega|\nabla v|^2dx}-
\delta \int_\Omega|\nabla v|^2dx\\
& =0.
\end{align*}
Retaining underlined terms on left hand side we get
\begin{align*}
& \int_\Omega |\Delta v|^2dx+ \delta\int_\Omega|\nabla v|^2dx-\int_\Omega a(x)v^2dx+ \int_\Omega \frac{f_1(x,u)}{u}v^2dx \\
&=\int_\Omega  \Big[~\underbrace{|\Delta v|^2}+\underline{\delta |\nabla v|^2}+\frac{4v}{u^2}\Delta u \nabla u.\nabla v-\underbrace{\frac{2v}{v}\Delta u\Delta v} \\
&-\frac{2}{u}|\nabla v|^2 \Delta u +\underbrace{\frac{v^2}{u^2}|\Delta u|^2}-\frac{2v^2}{u^3}|\nabla u|^2\Delta u
- \delta\underline{\frac{2v}{u}\nabla u \nabla v} +  \delta \underline{\frac{v^2}{u^2}|\nabla u|^2} \Big]dx  \\
&=\int_\Omega \Big[\Big(\Delta v-\frac{v}{u}\Delta u\Big)^2 + \delta \Big(\nabla v-\frac{v}{u}\nabla u\Big)^2+ \frac{4v}{u^2}\Delta u \nabla u\nabla v
- \frac{2}{u}|\nabla v|^2\Delta u-\frac{2v^2}{u^3}|\nabla u|^2 \Delta u \Big]dx.
\end{align*}
This implies that
\begin{align*}
& \int_\Omega |\Delta v|^2dx+ \delta\int_\Omega|\nabla v|^2dx-\int_\Omega a(x)v^2dx+\int_\Omega \frac{f_1(x,u)}{u}v^2dx\\
&\geq
\int_\Omega \Big[\delta\Big(\nabla v-\frac{v}{u}\nabla u\Big)^2+\frac{4v}{u^2}\Delta u \nabla u\nabla v - \frac{2}{u}|\nabla v|^2\Delta u-
\frac{2v^2}{u^3}|\nabla u|^2 \Delta u \Big]dx  \\
&= \int_\Omega \Big[\delta\Big(\nabla v-\frac{v}{u}\nabla u\Big)^2 -\frac{2}{u}\Delta u \Big( |\nabla v|^2+\frac{v^2}{u^2}|\nabla u|^2-\frac{2v \nabla u\nabla v  }{u}\Big)\Big]dx  \\
&= \int_\Omega \Big(\delta-\frac{2}{u}\Delta u\Big)\Big(\nabla v-\frac{v}{u}\nabla u\Big)^2dx\\
&\geq 0.
\end{align*}
Since $u$ satisfies $$\delta u\geq 2 \Delta u \,\,\,\mbox{in} \,\,\Om$$ so this implies that
\begin{align*}
\int_\Omega |\Delta v|^2dx+ \delta\int_\Omega|\nabla v|^2dx-\int_\Omega a(x)v^2dx+\int_\Omega \frac{f_1(x,u)}{u}v^2dx \geq 0
\end{align*}
and using the hypothesis \rm{(H1)}, we obtain
\begin{equation}
\int_\Omega |\Delta v|^2dx+ \delta\int_\Omega|\nabla v|^2dx-\int_\Omega a(x)v^2dx+\int_\Omega \frac{\partial f_1(x,u)}{\partial u} v^2dx\geq 0
\end{equation}
and therefore $u$ is stable. This completes the proof of this theorem.
\end{proof}


\begin{thebibliography}{0}

\bibitem{al1} W. Allegretto, Positive solutions and spectral properties of weakly coupled elliptic systems, J. Math. Anal. Appl. \textbf{120} (1986), no. 2, 723--729.
\bibitem{al2} W. Allegretto, On the principal eigenvalues of indefinite elliptic problems, Math. Z. \textbf{195} (1987), no. 1, 29--35.
\bibitem{al3} W. Allegretto, Sturmian theorems for second order systems. Proc. Amer. Math. Soc. 94 (1985), no. 2, 291--296.
\bibitem{ale} W. Allegretto and Y.X.Huang, A Picone's identity for the p-Laplacian and applications,  Nonlinear Anal.,\,\textbf{32}(7) (1998), pp. 819--830.
\bibitem{bal} K.Bal, Generalized Picone's identity and its applications,  Electron. J. Diff. Equations.,\,no. \textbf{243} (2013), pp. 1--6.


\bibitem{ber}  E. Berchio, A. Farina, A. Ferrero, F. Gazzola, Existence and stability of entire solutions to a semilinear fourth order elliptic problem,
J. Diff. Equations \textbf{252} (2012), 2596--2616.

\bibitem{bog} G. Bogn\'{a}r,  O. Do\v{s}l\'{y},  Picone-type identity for pseudo p-Laplacian with variable power, Electron. J. Diff. Equations 2012, No. \textbf{174}, 1--8.


\bibitem{cas} D. Castorina, P. Esposito and B. Sciunzi, Low dimensional instability for semilinear
and quasilinear problems in $\R^N,$  Commun. Pure Appl. Anal. \textbf{8} (2009), No. 6, 1779--1793

\bibitem{dun} D.\,R.Dunninger, A Picone integral identity for a class of fourth order elliptic differential inequalities,
Atti Accad. Naz. Lincei Rend. Cl. Sci. Fis. Mat. Natur.\,\textbf{50}(8) (1971), 630--641.



\bibitem{far} A. Ferrero, F. Gazzola and T. Weth, Positivity, symmetry and uniqueness for minimizers of second-order Sobolev inequalities,
 Ann. Mat. Pura Appl., \textbf{186} (2007), no. 4, 565--578.

\bibitem{gon} J.V.A. Goncalves, E. D. Silva, M. L. Silva, On positive solutions for a fourth order asymptotically linear elliptic equation under Navier boundary conditions,
  J. Math. Anal. Appl. \textbf{384}(2011)  387–399.


\bibitem{gazz}  F. Gazzola, H. Grunau and G. Sweers, Polyharmonic boundary value problems, A monograph on positivity preserving and
nonlinear higher order elliptic equations in bounded domain, Springer, 1991.

\bibitem{jar} J. Jaro\v{s},  The higher-order Picone identity and comparison of half-linear differential equations of even order,
Nonlinear Anal.,\,\textbf{74}(18) (2011), pp. 7513--7518.

\bibitem{kan} D.A. Kandilakis, M. Magiropoulos and  N.Zographopoulos,
Existence and bifurcation results for fourth-order elliptic equations involving two critical Sobolev exponents,
Glasg. Math. J. 51, no. 1 (2009), 127--141.


\bibitem{karat} J.\,Kar\'{a}tson, P.\,L.\,Simon,  On the linearized stability of positive solutions of quasilinear problems
with p-convex or p-concave nonlinearity, Nonlinear Anal.,\,\textbf{47}  (2001), pp. 4513--4520.


\bibitem{lin} A.E.Lindsay and L.Viga, Multiple quenching solutions of a fourth order parabolic pde with a singular nonlinearity modeling a MEMS capacitor,
SIAM J. Appl. Math., Vol. \textbf{72}, No. 3 (2012), 935--958.

\bibitem{you} Y.Maliki, On a $Q$-curvature equation on complete Riemannian manifolds,  Adv. Nonlinear Stud. \textbf{10} (2010), no. 1 (2010), 195--217.


\bibitem{manes} A. Manes, A.M. Micheletti, Un'estensione della teoria variazionale classica degli autovalori per operatori
ellittici del secondo ordine. Bollettino U.M.I., \textbf{7}, 1973, 285--301.


\bibitem{pel} J. A. Pelesko and D. H. Bernstein, Modeling MEMS and NEMS, Chapman $\&$ Hall CRC
Press, Boca Raton, FL, 2002.

\bibitem{sat} T.Sato, T.Watanabe, Singular positive solutions for a fourth order elliptic problem in $\R^N,$ Commun. Pure Appl. Anal. \textbf{10}, No. 1 (2011), 245--268.

\bibitem{str}  M. Struwe, Variational Methods: Applications to Nonlinear Partial Differential Equations and Hamiltonian
Systems, Fourth Edition, Springer, 2007.

\bibitem{tyagi}J.Tyagi, A nonlinear Picone's identity and applications, Applied Mathematics Letters, \textbf{26} (2013), 624--626.

\bibitem{tyagi1} J. Tyagi, Stability of positive solutions to $p \&2$-Laplace type equations, Diff. Equ.Appl. \textbf{5} (2013), 549--559.


\bibitem{tyagi2} J.Tyagi, A note on the stability of solutions to quasilinear elliptic equations, Advances in Cal. Var.,\textbf{ 6}, Issue 4 (2013),  483--492.

 \bibitem{war} G. Warnault, Liouville theorems for stable radial solutions for the biharmonic operator, Asymptot. Anal. \textbf{69} (2010) 87--98.



\bibitem{wei} J. Wei, X.\,Xu, Classification of solutions of higher order conformally invariant equations,
Math. Ann. \textbf{313} (1999), 207--228.

\bibitem{wei2} J. Wei, D.\,Ye, Liouville theorems for stable solutions of biharmonic problem,  Math. Ann. \textbf{356} (2013), No. 4, 1599--1612.

\bibitem{yosi}   N. Yoshida, Picone identities for half-linear elliptic operators with p(x)-Laplacians and applications to Sturmian comparison theory,
Nonlinear Anal. \textbf{74} (2011),no. 16, 5631--5642.




\end{thebibliography}
\end{document}